\documentclass{amsart} 
\author{Masaki Kameko}
\address{Department of Mathematical Sciences,
Shibaura Institute of Technology,
307 Minuma-ku Fukasaku, Saitama-City 337-8570, Japan}
\email{kameko@shibaura-it.ac.jp}
\thanks{This work was supported by JSPS KAKENHI Grant Number JP17K05263.}
\usepackage{amssymb,amscd, pb-diagram}
\usepackage[initials, alphabetic]{amsrefs}
\newtheorem{theorem}{Theorem}[section]
\newtheorem{proposition}[theorem]{Proposition}
\newtheorem{lemma}[theorem]{Lemma}

\theoremstyle{definition}

\title[Nilpotent elements]{
Nilpotent elements in the cohomology of the classifying space of a connected Lie group}

\newcommand{\spin}{\mathrm{Spin}}

\newcommand{\Tor}{\mathrm{Tor}}

\begin{document}\maketitle

\begin{abstract}
We give an example of a compact connected Lie group of the lowest rank such that the mod $2$ cohomology ring of its classifying space has a nonzero nilpotent element.
\end{abstract}

 
\section{Introduction}\label{section:1}


Let $p$ be a prime number and  $G$ a compact  Lie group. 
In \cite{quillen-1971}, Quillen defined
a homomorphism,
\[
q_G \colon H^{*}(BG;\mathbb{Z}/p)\to \lim_{\stackrel{\longleftarrow}{A\in \mathcal{A}}} H^{*}(BA;\mathbb{Z}/p),
\]
where $\mathcal{A}$ is the category of elementary abelian $p$-subgroups of $G$ and proved that $q_G$ is an $F$-isomorphism,
that is, each element in the kernel of $q_G$ is nilpotent and for each \[
y \in 
\displaystyle \lim_{\stackrel{\longleftarrow}{A\in \mathcal{A}}} H^{*}(BA;\mathbb{Z}/p), 
\]
there is a positive integer $n$ such that $y^{p^n}$ belongs to the image of $q_G$.
For $p=2$, $H^{*}(BA;\mathbb{Z}/2)$ is a polynomial ring. Hence, a nonzero element in the image of $q_G$ is not nilpotent. So, the nilradical
of $H^{*}(BG;\mathbb{Z}/2)$ is precisely the kernel of $q_G$. Thus, the above homomorphism $q_G$ is injective if and only if the mod $2$ cohomology ring 
$H^{*}(BG;\mathbb{Z}/2)$ has no nonzero nilpotent element.
In \cite{kono-yagita-1993},  Kono and Yagita showed that $q_G$ is not injective for $p=2$, $G=\mathrm{Spin}(11), E_7$
by showing the existence of a nonzero nilpotent element in $H^{*}(BG;\mathbb{Z}/2)$.
For an odd prime number $p$, Adams conjectured that $q_G$ is injective for all compact connected Lie groups.
Adams' conjecture remains an open problem.


On the other hand, for a compact connected Lie group $G$, a maximal torus $T$ exists. 
Let $W$ be the Weyl group $N(T)/T$.
We denote by $H^{*}(BT;\mathbb{Z})^W$ the ring of invariants of $W$.
We denote by $\Tor$ the torsion part of $H^{*}(BG;\mathbb{Z})$.  
Then, the inclusion map of $T$ induces a homomorphism,
\[
\iota_T^{*}\colon H^{*}(BG;\mathbb{Z})/\Tor\to H^{*}(BT;\mathbb{Z})^W.
\]
Borel showed that $\iota_T^*$  is injective.
In  \cite{feshbach-1981}, Feshbach gave a criterion for $\iota_T^*$ to be surjective, hence an isomorphism. 
In particular, after localized at $p$, $\iota_T^*$  is surjective if and only if
the $E_\infty$-term of the mod $p$ Bockstein spectral sequence of $BG$,
\[
H^{*}(BG;\mathbb{Z})/\Tor\otimes \mathbb{Z}/p,
\] 
has no nonzero nilpotent element. For $p=2$, Feshbach showed that for $G=\mathrm{Spin}(12)$, 
the $E_\infty$-term of the mod $2$ Bockstein spectral sequence of $BG$ has a nonzero nilpotent element.
As for spin groups, $\spin(n)$, Benson and Wood \cite{benson-wood-1995} computed the ring of invariants of the Weyl group
and they showed that $\iota_T^*$ is not surjective if and only if $n\geq 11$ and $n \equiv 3, 4, 5 \mod 8$.
However, as in the case of Adams' conjecture, for an odd prime number $p$, 
no example of a compact connected Lie group $G$ such that
the $E_\infty$-term of the mod $p$ Bockstein spectral sequence of $BG$ 
has a nonzero nilpotent element  is known. 


So, nonzero nilpotent elements in the cohomology of the classifying spaces of compact connected Lie groups are exciting subjects for study. 
However, 
no example of a compact connected Lie group $G$ such that $H^{*}(BG;\mathbb{Z}/2)$ has a nonzero nilpotent element
is known except for spin groups and the exceptional Lie group $E_7$. 
The purpose of this paper is to give a more straightforward  example to shed some light on the existence of nonzero nilpotent elements in the mod $2$ cohomology of 
the classifying space of a connected Lie group.

First, we define  a compact connected Lie group $G$. 
Let us consider the three fold product $SU(2)^3$ of the special unitary groups $SU(2)$.
Its center is an elementary abelian $2$-group $(\mathbb{Z}/2)^3$.
Let  $\Gamma$ be the kernel of the group homomorphism
$\varphi\colon (\mathbb{Z}/2)^3 \to \mathbb{Z}/2$ defined by $\varphi(a_1, a_2, a_3)=a_1a_2a_3$.
We define $G$ to be $SU(2)^3/\Gamma$.

Next, we state our results, saying that $G=SU(2)^3/\Gamma$ satisfies the required conditions.
Since $SU(2)^3/(\mathbb{Z}/2)^3=SO(3)^3$, 
we have the following fiber sequence:
\[
B\mathbb{Z}/2 \to BG \stackrel{\pi}{\longrightarrow} BSO(3)^3.
\]
Let $\pi_i\colon BSO(3)^3 \to BSO(3)$ be the projection onto the $i^{\mathrm{th}}$ factor.
The mod $2$ cohomology ring of $BSO(3)$ is given by
\[
H^*(BSO(3);\mathbb{Z}/2)=\mathbb{Z}/2[w_2, w_3],
\]
where $w_i$ is the universal $i$-th Stiefel-Whitney class for $i=2, 3$.

Let
$w_k'=\pi^*(\pi_1^{*}(w_k))$ and $w_k''=\pi^*(\pi_2^{*}(w_k))$.
Let $u_{16}$ be the Stiefel-Whitney class $w_{16}(\rho)$ of a real representation $\rho\colon G\to O(16)$.
We will give the definition of $\rho$ in Section~\ref{section:2}. 
Let $f_5$, $f_9$, $g_4$, $g_7$, $g_8$  be  polynomials 
defined by
\begin{align*}
f_5&=w_2'w_3''+w_2''w_3', 
\\
f_9& =w_3'^2 w_3''+w_3''^2w_3',
\\
g_4&=w_2'w_2'',
\\
g_7&=w_2'w_2''(w_3'+w_3''),
\\
g_8 &=w_3'w_3''(w_2'+w_2''),
\end{align*}
respectively. Then, our results are stated as follows:


\begin{theorem}
\label{theorem:1.1}
The mod $2$ cohomology ring of $BG$ is 
\[
\mathbb{Z}/2[ w_2', w_2'', w_3', w_3'', u_{16}]/(f_5, f_9)
\]
and its nilradical  is generated by
$
g_7$, $g_8$.
\end{theorem}


\begin{theorem}
\label{theorem:1.2}
The $E_\infty$-term of the mod $2$ Bockstein spectral sequence  of $BG$ is
\[
\mathbb{Z}/2[w_2'^{2}, w_2''^{2}, u_{16}] \otimes \Delta(g_4, g_8),
\]
where
$\Delta(g_4, g_8)$ is the vector space over $\mathbb{Z}/2$ spanned by $1$, $g_4$, $g_8$ and $g_4g_8$.
Its nilradical is generated by $g_8$. \end{theorem}

The computations involved in these theorems are similar to those of Quillen in \cite{quillen-1971-2} and Kono in \cite{kono-1986}. We have no claim for novelty in this respect.

The rank of $SU(2)^3/\Gamma$  is $3$.
If the rank of a compact connected Lie group  is lower than $3$, then it  is homotopy equivalent to 
one of $T$, $SU(2)$, $T^2$, $T \times SU(2)$, $SU(2)\times SU(2)$, $SU(3)$, $G_2$ or their quotient groups 
by their central subgroups.
For such a compact connected Lie group, the mod $2$ cohomology ring of
its classifying space is a polynomial ring so that it has no nonzero nilpotent element.
Thus $SU(2)^3/\Gamma$ is a lowest rank Lie group such that the mod $2$ cohomology of its classifying space has a nilpotent element.

We hope our results  shed some light on Adams' conjecture since, contrary to spin groups, we have an odd prime analog of the group $SU(2)^3/\Gamma$.
Let  $\Gamma_2$ be the kernel of the determinant homomorphism $\det\colon (S^1)^3 \to S^1$.
Consider the quotient group.
$
U(p)^3/\Gamma_2.
$
It is the odd prime counterpart 
as 
the group $U(2)^3/\Gamma_2$ is the central extension of the group $SU(2)^3/\Gamma$ by $S^1$.
But that is another story and 
we wish to deal with the group $U(p)^3/\Gamma_2$  in another paper.

In what follows, we assume that $G$ is the compact connected Lie group $SU(2)^3/\Gamma$.
We also  denote the mod $2$ cohomology ring of $X$ by $H^{*}(X)$ rather than $H^{*}(X;\mathbb{Z}/2)$.
This paper is organized as follows: In Section~\ref{section:2}, we compute the Leray-Serre spectral sequence associated with the fiber sequence
\[
B\mathbb{Z}/2 \stackrel{\iota}{\longrightarrow} BG \stackrel{\pi}{\longrightarrow}  BSO(3)^3
\]
to describe the mod $2$ cohomology ring $H^{*}(BG)$ and prove Theorem~\ref{theorem:1.1}.
In Section~\ref{section:3}, we define and compute the $Q_0$-cohomology of $H^{*}(BG)$ to complete the proof of 
Theorem~\ref{theorem:1.2}.

The author would like to thank the referee for his kind and helpful comments and suggestions. They improved the presentation of this paper considerably.


\section{The mod $2$ cohomology ring}\label{section:2}

In this section, we compute the mod $2$ cohomology ring of $BG$ by the Leray-Serre spectral sequence associated with the fiber sequence
\[
B\mathbb{Z}/2 \stackrel{\iota}{\longrightarrow} BG \stackrel{\pi}{\longrightarrow}  BSO(3)^3.
\]


First, we recall the mod $2$ cohomology rings of $BSO(3)$ and $BSO(3)^3$.
As stated in Section~\ref{section:1}, the mod $2$ cohomology ring is given by
\[
H^{*}(BSO(3);\mathbb{Z})=\mathbb{Z}/2[ w_2, w_3].
\]
Let 
$Q_i$ be the Milnor operation 
\[
Q_i\colon H^k(X) \to H^{k+2^{i+1}-1}(X)
\]
defined inductively by 
\[
Q_0=\mathrm{Sq}^1, \quad Q_{i+1}= \mathrm{Sq}^{2^{i+1}} Q_i + Q_i \mathrm{Sq}^{2^{i+1}}
\]
for $i\geq 0$. 
The Wu formula yields 
\begin{align*}
Q_0(w_2)&=w_3, 
\\
Q_1(w_2)&=w_2w_3, 
\\Q_2(w_2)&= w_2^3 w_3+w_3^3.
\end{align*}
Recall that $\pi_i\colon BSO(3)^3\to BSO(3)$ ($i=1,2,3$) is the projection onto the $i^{\mathrm{th}}$ factor. By abuse of notation, 
we define  elements $w_k'$, $w_k''$, $w_k'''$  ($k=2, 3$) in $H^{*}(BSO(3)^3)$ by
$
w_k'=\pi_1^*(w_k)$, 
$w_k''=\pi_2^*(w_k)$, 
$w_k'''=\pi_3^*(w_k).
$
Let us define elements $v_2$, $v_3$ by
\begin{align*}
v_2&=w_2'+w_2''+w_2''',
\\
v_3&=w_3'+w_3''+w_3''', 
\end{align*}
and
ideals $I_1$, $I_2$ by
\begin{align*}
I_1&=(v_2, v_3), 
\\
I_2&=(v_2, v_3, Q_1(v_2)).
\end{align*}
Again, by abuse of notation, let
\begin{align*}
f_5&=w_2'w_3''+w_2''w_3', \\
f_9&=w_3'^2w_3''+w_3''^2w_3' \in H^{*}(BSO(3)^3).
\end{align*}
Then, by direct calculations, we have
\begin{align*}
Q_0 v_2&=v_3, 
\\
Q_1 v_2&\equiv f_5 \quad  \mod I_1,
\\
Q_2v_2&\equiv  f_9 \quad  \mod I_2.
\end{align*}


Now, we compute the Leray-Serre spectral sequence.
The $E_2$-term is given by 
\[
E_2^{p,q}=H^p(BSO(3)^3) \otimes H^q(B\mathbb{Z}/2),
\]
so that
\[
E_2=\mathbb{Z}/2[w_2', w_2'', v_2, w_3', w_3'', v_3, u_1],
\]
where $u_1$ is the generator of $H^{1}(B\mathbb{Z}/2)\cong \mathbb{Z}/2$.
A possible first nontrivial differential is $d_2$. 
Let $\iota_i\colon SU(2) \to SU(2)^3$ be the inclusion map to the $i^{\mathrm{th}}$ factor,
\[
\iota_{1}(g)=(g, 1, 1),\quad  \iota_{2}(g)=(1, g, 1), \quad \iota_{3}(g)=(1, 1, g).
\]
Then, they induce the following commutative diagram.
\[
\begin{diagram}
\node{B\mathbb{Z}/2} \arrow{e,t}{=} 
\arrow{s} \node{B\mathbb{Z}/2}
\arrow{s} 
\\
\node{BSU(2)} \arrow{e,t}{\iota_i} \arrow{s} \node{BG} \arrow{s} 
\\
\node{BSO(3)} \arrow{e,t}{\iota_i} \node{BSO(3)^3.}
\end{diagram}
\]
Since the differential $d_2$ in the Leray-Serre spectral sequence associated with 
the left column homotopy fibration
 is 
 \[
 d_2(u_1)=w_2, 
 \]
 we have 
 \[
 d_2(u_1)=v_2
 \]
 in the Leray-Serre spectral sequence for the right column homotopy fibration.

To compute the higher differentials, we consider the following diagram.
Let $K(\mathbb{Z}/2,2)$ be the Eilenberg-MacLane space.
Let 
\[
k\colon BSO(3)^3 \to K(\mathbb{Z}/2, 2)
\]
 be a map representing the cohomology class $v_2\in H^2(BSO(3)^3)$ such that
\[
k^*(u_2)=v_2
\]
 where $u_2$ is the generator of $H^{2}(K(\mathbb{Z}/2, 2))\cong \mathbb{Z}/2$.
 Putting the path space fibration over $K(\mathbb{Z}/2,2)$ in the right column, we have the following commutative diagram.
\[
\begin{diagram}
\node{B\mathbb{Z}/2} \arrow{e,t}{=} 
\arrow{s} \node{B\mathbb{Z}/2} \arrow{s}
\\
 \node{BG} \arrow{s} \arrow{e} \node{PK(\mathbb{Z}/2,2)} \arrow{s}
\\
\node{BSO(3)^3}\arrow{e,t}{k} \node{K(\mathbb{Z}/2,2).}
\end{diagram}
\]
The mod $2$ cohomology rings and
the Leray-Serre spectral sequence for the path space fibration are known. We refer the reader to Serre's classical paper \cite{serre-1953}.
Its $E_2$-term is 
\[
E_2=\mathbb{Z}/2[ u_2, \mathrm{Sq}^1 u_2, \mathrm{Sq}^2 \mathrm{Sq}^1 u_2, \dots ] \otimes \mathbb{Z}/2[ u_1] 
\]
and nontrivial differentials are given by
\[
d_{2^n+1}(u_1^{2^n})=\mathrm{Sq}^{2^{n-1}}\cdots \mathrm{Sq}^1 u_2
\]
for $n\geq 0$.
 

\begin{lemma}\label{lemma:2.1}
For $x\in H^{2}(X)$ and $k\geq 1$, we have
\[
Q_k(x)= \mathrm{Sq}^{2^k}\cdots \mathrm{Sq}^{2^0} (x).
\]
\end{lemma}
\begin{proof}
We prove this lemma by induction on $k$.
Suppose  $k=1$. By the unstable condition, we have $\mathrm{Sq}^2(x)=x^2$. By the Cartan formula, we have $\mathrm{Sq}^1 (x^2)=0$.
Hence, we have
\begin{align*}
Q_{1}(x)&=\mathrm{Sq}^{2} Q_{0} (x)+ Q_{0}{\mathrm{Sq}^2}(x)
\\
&=\mathrm{Sq}^{2} Q_0 (x). 
\end{align*}
For $k \geq 2$, 
by the definition of $Q_{i+1}$ and the unstable condition, we have
\begin{align*}
Q_{k}(x)&=\mathrm{Sq}^{2^{k}} Q_{k-1} (x)+ Q_{k-1} \mathrm{Sq}^{2^{k}} (x)
\\
&=\mathrm{Sq}^{2^{k}} Q_{k-1} (x). \qedhere
\end{align*}
\end{proof}

From $d_2(u_1)=v_2$ and  the action of $Q_0$, $Q_1$, $Q_2$ on $v_2$, by Lemma~\ref{lemma:2.1} and the Leray-Serre spectral sequence for the 
above path space fibration, 
we have 
\begin{align*}
d_3(u_1^2)&= v_3, \\
d_5(u_1^4)&=f_5, \\
d_9(u_1^8)&=f_9.
\end{align*}
It is easy to see that
\begin{align*}
E_3&=\mathbb{Z}/2[w_2', w_2'', w_3', w_3'', v_3, u_1^2], 
\\
E_4&=\mathbb{Z}/2[w_2', w_2'', w_3', w_3'', u_1^4], 
\\
E_6&=\mathbb{Z}/2[w_2', w_2'', w_3', w_3'', u_1^8]/(f_5), 
\end{align*}
In $\mathbb{Z}/2[w_2', w_2'', w_3', w_3'']$, we consider the sequence
$
f_5$, 
$f_9
$.
It is a regular sequence since their greatest common divisor is $1$.
 Therefore, we have 
 \[
 E_{10}=\mathbb{Z}/2[w_2', w_2'', w_3', w_3'', u_1^{16}]/(f_5,f_9).
 \]
 

To prove that the spectral sequence collapses at the $E_{10}$-term, we consider the Stiefel-Whitney class of a 
 real representation 
 \[
 \rho\colon G \to O(16)
 \]
 defined as follow.
 On the one hand, 
since $\mathbb{C}$ is isomorphic to  $\mathbb{R}^2$ as a vector space over $\mathbb{R}$, 
$\mathbb{C}^2$ is isomorphic to $\mathbb{R}^4$. Then, the tautological representation of $SU(2)$ on $\mathbb{C}^2$
induces the inclusion map
\[
j\colon SU(2)\to SO(4).
\]
On the other hand, we have an isomorphism
\[
SO(4)=SU(2)\times_{\mathbb{Z}/2} SU(2).
\]
Since 
\[
G=SU(2) \times_{\mathbb{Z}/2} (SU(2)\times_{\mathbb{Z}/2} SU(2)) = SU(2) \times_{\mathbb{Z}/2}  SO(4), 
\]
we may regard $G$ as a  subgroup of 
\[
SO(4) \times_{\mathbb{Z}/2} SO(4)
\]
with the inclusion map induced by 
\[
j\times 1\colon SU(2) \times SO(4)\to SO(4)\times SO(4).
\]
Let 
\[
\varphi\colon SO(4) \times SO(4) \to O(16)
\]
be the real representation given by 
\[
(g_1, g_2) m= g_1 m g_2^{-1}
\]
where $(g_1, g_2) \in SO(4) \times SO(4)$ and $m$ is a $4\times 4$ matrix with real coefficients.
Then, $\varphi$ induced a $16$-dimensional real representation. 
\[
\varphi'\colon SO(4)\times_{\mathbb{Z}/2} SO(4) \to O(16).
\]
We define the representation $\rho$ as the restriction of $\varphi'$ to $G$.


 \begin{proposition}\label{proposition:2.2}
 The Stiefel-Whitney class $w_{16}(\rho)$ of the real representation $\rho$ is indecomposable in  $H^{*}(BG)$.
 It is 
 represented by $u_1^{16}$ in the Leray-Serre spectral sequence associated with the fiber sequence
 \[
 B\mathbb{Z}/2\to BG \to BSO(3)^3.
 \]
 \end{proposition}
 
 \begin{proof}
 Let $\iota\colon \mathbb{Z}/2\to G$ be the inclusion map of the center of $G$. We may regard $\mathbb{Z}/2$ as a subgroup of the center of 
 $SU(2)^3$. Thus, the inlcusion map $\iota$ factors through the projection $SU(2)^3\to G$.

The restriction of $\rho$ to the center of $G$ is $16\lambda$ where $\lambda$ is the nontrivial $1$ dimensional real representation
of $\mathbb{Z}/2$.
So, the Stiefel-Whitney class $w_{16}(\rho \circ \iota)$ is nonzero.
If $d_r(u_1^{16})\not=0$ for some $r$, up to degree $\leq 16$, $H^{*}(BG)$ is generated by $w_2', w_2'', w_3', w_3''$. 
However, since 
$\iota$ factors through $SU(2)^3$, and since $BSU(2)^3$ is $3$-connected, the induced homomorphism sends $w_2', w_2'', w_3', w_3''$ to zero.
So, $w_{16}(\rho\circ \iota)$ is zero. It is a contradiction.
Therefore, $u_1^{16}$ is a permanent cycle in the Leray-Serre spectral sequence and it is represented by $w_{16}(\rho)$.
 \end{proof}

By Propositino~\ref{proposition:2.2}, the spectral sequence collapses at the $E_{10}$-term, that is, $E_\infty=E_{10}$ and 
we obtain the  first half of Theorem~\ref{theorem:1.1}.


\begin{proposition}
\label{proposition:2.3}
We have 
\[ H^{*}(BG)=\mathbb{Z}/2[w_2', w_2'', w_3', w_3'', u_{16}]/(f_5, f_{9}),
\]
where $u_{16}$ is the Stiefel-Whitney class $w_{16}(\rho)$.
\end{proposition}

To prove the second half of Theorem~\ref{theorem:1.1}, let us define a ring homomorphism
\[
\eta \colon \mathbb{Z}/2[w_2', w_2'', w_3', w_3'', u_{16}]\to \mathbb{Z}/2[w_2', w_2'', u, u_{16}]
\]
by $\eta(w_2')=w_2'$, $\eta(w_2'')=w_2''$, $\eta(w_3')=w_2' u$, $\eta(w_3'')=w_2''u$, $\eta(u_{16})=u_{16}$.
It induces the following ring homomorphism
\[
\eta'\colon H^{*}(BG)\to \mathbb{Z}/2[w_2', w_2'', u , u_{16}]/(u^3 w_2'w_2''(w_2'+w_2'')).
\]
Let 
\[
R_0=\mathbb{Z}/2[w_2', w_2'', w_3', w_3'', u_{16}].
\]
From Proposition~\ref{proposition:2.3}, using the fact that $f_5$, $f_9$ is a regular sequence in 
$R_0$, 
the Poincar\'{e} series of $H^{*}(BG)$ is given by
\[
PS(H^{*}(BG), t)=\dfrac{(1-t^5)(1-t^9)}{(1-t^2)^2(1-t^3)^2(1-t^{16})}.
\]
On the other hand, it is also easy to see that
the image of $\eta'$  is 
spanned by monomials
\[
u^\ell w_2'^m w_2''^nu_{16}^k, 
\]
where $k$ ranges over all non-negative integers, for $\ell=0,1,2$, $(m, n)$ satisfies the condition $m+n\geq \ell$, 
and for $\ell\geq 3$, $(m,n)$ satisfies one of the following conditions:
$m\geq \ell$, $n=0$ or $m=1$, $n\geq \ell-1$ or $m=0$, $n\geq \ell$.
Thus, the Poincar\'{e} series $PS(\mathrm{Im}\, \eta', t)$ 
is
\[
\dfrac{1}{1-t^{16}}\left(  \dfrac{1}{(1-t^2)^2}+ t\left( \dfrac{1}{(1-t^2)^2}-1\right)+t^2 \left( \dfrac{1}{(1-t^2)^2}-1-2t^2\right)
+\sum_{\ell=3}^\infty\dfrac{3t^{3\ell}}{1-t^2}\right).
\]
Then, we have
\[
PS(H^{*}(BG),t)=PS(\mathrm{Im}\, \eta', t).
\]
Thus,  $\eta'$
is injective. In view of this injective homomorphism $\eta'$, 
it is easy to see that elements $g_7$, $g_8$ corresponding to $u w_2' w_2'' (w_2'+w_2'')$, $u^2  w_2' w_2'' (w_2'+w_2'')$, respectively,   are
nilpotent. 
So we obtain the following second half of Theorem~\ref{theorem:1.1}.


\begin{proposition}\label{proposition:2.4}
The nilradical  of $H^{*}(BG)$ is the ideal generated by
two elements 
$g_7$ and $g_8$.
\end{proposition}


\section{The mod $2$ Bockstein spectral sequence}\label{section:3}

For each $ i\geq 0$, we have $Q_i Q_i=0$. Hence, for a graded vector space $M$ over $\mathbb{Z}/2$ with $Q_i$-action, we may define 
$Q_i$-cohomology $H^{*}(M, Q_i)$ by 
\[
\mathrm{Ker}\, Q_i /\mathrm{Im}\, Q_i.
\]
In particular, the $E_2$-term of the mod $2$ Bockstein spectral sequence of $BG$ is the $Q_0$-cohomology $H^*(H^{*}(BG), Q_0)$.
In this section, to show that the mod $2$ Bockstein spectral sequence of $BG$
collapses at the $E_2$-term, we compute 
the $Q_0$-cohomology of the mod $2$ cohomology of $BG$, i.e. 
\[
H^{*}(H^{*}(BG),Q_0)=\mathrm{Ker}\, Q_0/\mathrm{Im}\, Q_0.
\]

First, we recall the action of $Q_0$ on $H^{*}(BG)$. The action of $Q_0$ on $w_2', w_2'', w_3', w_3''$ is clear from that on $H^{*}(BSO(3))$. 
We need to determine the action of $Q_0$ on $u_{16}$.


\begin{proposition}
In $H^{*}(BG)$, we have $Q_0(u_{16})=0$.
\end{proposition}

\begin{proof}
The generator $u_{16}$ is defined as the Stiefel-Whitney class $w_{16}(\rho)$ of the $16$-dimensional real representation $\rho\colon G\to O(16)$.
Hence, $w_{17}(\rho)=0$. Since $BG$ is simply-connected, we also have $w_1(\rho)=0$.
By the Wu formula, we have $\mathrm{Sq}^1 w_{16}(\rho)=w_{17}(\rho)+w_1(\rho)w_{16}(\rho)$. Therefore, we have the desired result.
\end{proof}
Let 
\[
R_0=\mathbb{Z}/2[w_2', w_2'', w_3', w_3'', u_{16}].
\]
We consider the action of $Q_0$ on $w_2'$, $w_2''$, $w_3'$, $w_3''$, $u_{16}$ in $R_0$.
It is given by
\[
Q_0(w_2')=w_3', \quad Q_0(w_2'')=w_3'', \quad Q_0(w_3')=0, \quad Q_0(w_3'')=0, \quad Q_0(u_{16})= 0.
\]
Let 
\[
R_1=R_0/(f_5),
\quad 
R_2=R_0/(f_5, f_9).
\]
It is clear that $R_2=H^{*}(BG)$ and $H^*(H^{*}(BG), Q_0)=H^*(R_2, Q_0)$. 
We will prove the following Proposition~\ref{proposition:3.2} at the end of this section.


\begin{proposition}
\label{proposition:3.2}
We have 
\[
H^{*}(R_2, Q_0)=\mathbb{Z}/2[w_2'^2, w_2''^2, u_{16}]\otimes \Delta(g_4, g_8).
\]
\end{proposition}

The $E_1$-term of the mod $2$ Bockstein spectral sequence of $BG$  is the mod $2$ cohomology ring of $BG$
and $d_1$ is $Q_0$.
Since, by Proposition~\ref{proposition:3.2},  the $E_2$-term  has no nonzero odd degree element, 
the spectral sequence collapses at the $E_2$-term. It is also clear that $g_4^2=w_2'^2w_2''^2\not=0$, $g_8^2=0$ from
Theorem~\ref{theorem:1.1}. Hence, we obtain Theorem~\ref{theorem:1.2}.

Now, we complete the proof of Theorem~\ref{theorem:1.2} by proving Proposition~\ref{proposition:3.2}.


\begin{proof}[Proof of Proposition~\ref{proposition:3.2}]
We start with $H^*(R_0, Q_0)$.
It is clear that
\[
H^*(R_0,Q_0)=\mathbb{Z}/2[w_2'^2, w_2''^2, u_{16}].
\]
We denote by $(-)\times a$ the multiplication by $a$.
Consider a short exact sequence
\[
0 \to R_0 \stackrel{(-)\times f_5}{\longrightarrow} R_0 \to R_1 \to 0.
\]
Since $Q_0$ is a derivation and $Q_0 f_5=0$, 
$Q_0$ commutes with $(-)\times f_5$. Hence, this short exact sequence induces
a long exact sequence in $Q_0$-cohomology:
\[
\cdots   \to H^{i}(R_0,Q_0) \to H^{i}(R_1, Q_0) \stackrel{\delta_4}{\longrightarrow} H^{i-4}(R_0,Q_0) \to \cdots
\]
Since $H^{odd}(R_0,Q_0)=0$, this long exact sequence splits into short exact sequences:
\[
0 \to H^{2i}(R_0;Q_0) \to H^{2i}(R_1, Q_0) \stackrel{\delta_4}{\longrightarrow} H^{2i-4}(R_0,Q_0) \to 0
\]
and $H^{odd}(R_1, Q_0)=0$. Since $Q_0 g_4=f_5$ in $R_0$, 
$g_4$ is nonzero in $R_1$ and 
$\delta_4(g_4)=1$. Therefore, we have 
\[
H^{*}(R_1,Q_0)=\mathbb{Z}/2[w_2'^2, w_2''^2, u_{16}]\otimes \Delta(g_4).
\]
Next, let us consider a short exact sequence
\[
0 \to R_1 \stackrel{(-)\times f_9}{\longrightarrow} R_1 \to R_2 \to 0.
\]
Again, since $Q_0f_9=0$ and $Q_0$ is a derivation, it induces a long exact sequence in $Q_0$-cohomology.
As above, since $H^{odd}(R_1, Q_0)=\{ 0\}$, we have short exact sequences
\[
0 \to H^{2i}(R_1;Q_0) \to H^{2i}(R_2, Q_0) \stackrel{\delta_8}{\longrightarrow} H^{2i-8}(R_1,Q_0) \to 0
\]
and $H^{odd}(R_2, Q_0)=\{0\}$.
Since  $Q_0g_8=f_9$, we obtain the desired result
\[
H^{*}(R_2, Q_0)=\mathbb{Z}/2[w_2'^2, w_2''^2, u_{16}]\otimes \Delta(g_4, g_8). \qedhere
\]
\end{proof}


\end{document}